\documentclass[10pt]{article}
\usepackage[utf8]{inputenc}
\usepackage{t1enc}  
\usepackage{amsfonts}  
\usepackage{amsmath, amssymb, graphics, setspace}
\usepackage{graphicx}
\usepackage{framed}
\usepackage{epstopdf}
\usepackage{verbatim}
\usepackage{enumerate}
\usepackage{amsthm}
\usepackage{hyperref}
\usepackage{url}
\usepackage{tikz}

\newcommand{\cS}{\mbox{$\mathcal S$}}

\newcommand{\cG}{\mbox{$\mathcal G$}}
\newcommand{\cH}{\mbox{$\mathcal H$}}

\newcommand{\cN}{\mbox{$\mathcal N$}}

\newcommand{\cL}{\mbox{$\mathcal L$}}

\newcommand{\cD}{\mbox{$\mathcal D$}}

\newcommand{\cR}{\mbox{$\mathcal R$}}

\newcommand{\K}{\mathbb K}
\newcommand{\N}{\mathbb N}

\newcommand{\F}{\mathbb F}

\let\N\relax
\DeclareMathOperator{\N}{N}
\DeclareMathOperator{\Tr}{Tr}
\let\NG\relax
\DeclareMathOperator{\NG}{NG}
\DeclareMathOperator{\PG}{PG}
\DeclareMathOperator{\ex}{ex}

\newtheorem{thm}{Theorem}[section]
\newtheorem{lemma}[thm]{Lemma}
\newtheorem{cor}[thm]{Corollary}

\newtheorem{prop}[thm]{Proposition}
\newtheorem{conjecture}{Conjecture}

\title{Singer difference sets and the projective norm graph}
\author{Tam\'as
M\'esz\'aros\footnotemark[1]\footnote{Freie Universit\"at Berlin} \footnote{Position funded by the DRS Fellowship
Program of Freie Universit\"at Berlin} \and Lajos
R\'onyai\footnote{Institute for Computer Science and Control, Hungarian
Academy of Sciences and BME} \footnote{Research supporten in part by NKFIH Grant No. K115288.} \and Tibor Szab\'o\footnotemark[1] \footnote{Research
supported in part by GIF grant No. G-1347-304.6/2016.}}

\begin{document}

\maketitle

\begin{abstract} 
We 
demonstrate a close connection between the classic planar Singer difference
sets and certain norm equation systems arising from projective norm
graphs. This, on the one hand leads to a novel description of planar
Singer difference sets as a subset $\cH$ of $\cN$, 
the group of elements of norm 1 in the field extension $\mathbb
F_{q^3}/\mathbb F_q$. ${\cal H }$ is given as the solution set of a simple
polynomial equation, and we obtain an explicit formula expressing each 
non-identity element
of ${\cal N}$ as a product $B\cdot C^{-1}$
with $B, C\in \cH$.
The description and the definitions naturally carry over to the
nonplanar and the infinite setting. On the other hand, relying heavily on the difference set properties, we 
also complete the proof that the projective norm graph $\NG( q, 4)$
does contain the complete bipartite graph $K_{4,6}$ for every prime power $q \geq 5$. 
This complements the property, known for more than two decades, 
that projective norm graphs do {\em not} contain $K_{4,7}$ (and hence
provide tight lower bounds for the Tur\'an number $\ex(n, K_{4,7})$). 
\end{abstract}

\section{Introduction and results} \label{sec:intro}

\subsection{Projective norm graphs}\label{subs:normgraph}

Let $\F$ be an arbitrary field, for $t\geq 2$ let $\K$ be a cyclic
Galois extension of degree $t-1$ and let us denote by $\N$ the norm of
this extension, i.e. for $A\in \K$ we have 
$\N(A)=A\cdot \phi(A)\cdot \phi^{(2)}(A)\cdots \phi^{(t-2)}(A)$ where the automorphism $\phi$ generates the Galois group of $\mathbb
K/\mathbb F$ and $\phi^{(j)}$
denotes the $j$-fold iteration of $\phi$. Then the {\em projective
  norm-graph} $\NG(\F,\K)$ has vertex set $\K \times \F^*$, where $\F^*$
denotes the multiplicative subgroup of $\F$, and
two vertices $(A,a)$ and $(B,b)$ are adjacent 
if and only if $\N(A+B)=ab$.\footnote{For technical reasons we allow the two vertices to be the same, i.e. we allow loop edges.}.

Projective norm graphs over finite fields were introduced by Alon, R\'onyai and
Szab\'o~\cite{ARSz} in connection with the Tur\'an problem for
complete bipartite graphs. 
For a prime power $q=p^k$ we use the standard notation $\F_{q}$ for the finite 
field with $q$ elements. When $\F=\F_q$, then $\K =\F_{q^{t-1}}$ and the automorphism
$\phi$ above can be chosen to be 
the Frobenius automorphism $\phi(X) = X^q$ and the graph $\NG(\F_{q},\F_{q^{t-1}})$
is denoted by $\NG(q,t)$. It is not difficult to
show that $\NG(q,t)$ has $\frac{1}{2} q^{t-1} (q-1) (q^{t-1} -1) \approx
\frac{1}{2} n^{2-\frac{1}{t}}$ edges, where $n = q^{t-1}(q-1)$ is the number
of vertices. In \cite{ARSz} it was also proved that $\NG(q,t)$
does not contain $K_{t,(t-1)!+1}$ as a subgraph, and hence has
essentially the highest number of edges among graphs with this
property, on the same number of vertices. 
 
Determining the largest number $\ex(n, H)$ of edges a
graph on $n$ vertices without a subgraph isomorphic to $H$ can have is 
one of the classic problems of extremal graph theory, its history going back more
than a century, to the theorem of Mantel about triangle-free
graphs. The value of $\ex(n, H)$ is settled
asymptotically when $H$ is non-bipartite, but for bipartite graphs its order of magnitude
is known only in a handful of cases. Even for the simplest
bipartite graphs, such as even cycles and complete bipartite graphs,
the question is wide open. The general upper bound of K\H ov\'ari, S\'os and
Tur\'an~\cite{KST} states that $\ex(n, K_{t,s}) \leq
c_{s}n^{2-\frac{1}{t}}$ for every $t\leq s$. Matching lower bounds are
known only for $t=2$ (Klein~\cite{E}), $t=3$ (Brown~\cite{Br}),
and by the projective
norm graphs for arbitrary $t\geq 2$ and $s\geq (t-1)!+1$. 
The fundamental question of the order of magnitude of $\ex(n, K_{t,s})$
is very much open for any $4\leq t \leq s \leq (t-1)!$. 

In general it is not known how large complete bipartite graphs 
the projective norm graph $\NG(q,t)$ contains. For $t\geq 4$ it could
even be the case that there is an infinite sequence of prime powers $q$ such
that $\NG(q,t)$ does not contain a copy of $K_{t,t}$ and hence resolves the
question of the order of magnitude of $\ex(n, K_{t,t})$ for every $t$
and $s$. With that, the determination of the largest integer $s(t)$, for
which $\NG (q,t)$ contains a $K_{t,s(t)}$ for every large enough prime
power $q$, has potentially far reaching consequences. 
For $t=2$ and $t=3$ the projective norm graph $\NG(q,t)$ does contain 
a $K_{t, (t-1)!}$ by combinatorial reasons (the KST upper bound), so
$s(2)=1$ and $s(3)=2$. For $t\geq 4$ however, it was only known that 
$t-1 \leq s(t) \leq (t-1)!$. 

In this direction Grosu~\cite{G} has recently shown that $\NG(p,4)$
contains a copy of the complete bipartite graph $K_{4,6}$ for
roughly $\frac{1}{9}$-fraction of all primes $p$.
In~\cite{BMRSz}, among other things, this was extended
for any prime power $q$ if the characteristic $p$ is not $2$ or $3$.
Furthermore, the proof also provided many copies of $K_{4,6}$.
Computer calculations have also suggested that the 
same holds in the case $p\in \{2,3\}$ as well, but the arguments 
in~\cite{BMRSz} crucially used the restriction on the
characteristic. In this paper we provide different arguments to show 
the existence of $K_{4,6}$ in $NG(q,4)$ for the cases when $p\not\equiv 1$
$\pmod 3$ and $q\geq 5$, and hence establish $s(4)=6$. In the process we
uncover a close connection between the norm equation systems arising
from the projective norm graph and the classic Singer difference
sets. 
We consider this connection one of the main contributions of our paper. 
In the next subsection we introduce the necessary background
for the latter.

\subsection{Difference sets}\label{subs:diffsets}

Given a multiplicative group $\cG$, a subset $\cD \subseteq \cG$ is
called a {\em planar difference set} if every non-identity element
$A\in \cG$ has a unique representation
as a product
of an element from $\cD$ and an element from $\cD^{-1}$, where
$\cD^{-1} : = \{ d^{-1} : d\in \cD\}$ denotes 
the set of inverses of the elements of $\cD$. 
We refer to this representation as the 
{\em mixed representation of $A$ with respect to $\cD$}.

Difference sets in finite groups are central and diverse objects in
design theory with a rich history and numerous applications both inside
and outside mathematics. For a gentle introduction and survey the
reader may consult e.g. \cite{MP}. 
If a finite group $\cG$ admits a planar difference set of size $m$,
then its order,  by simple counting, must be of the form
$\ell^2+\ell+1$, where $\ell = m-1$. Planar difference sets in Abelian
groups are only known to exist if $\cG$ is a cyclic group and $\ell$ is a prime power. 
In what follows we will simply write 'difference set' instead of
planar difference set. 
The first construction was given by Singer~\cite{S}, using the finite 
projective plane $\PG(2,q)$. 
A \emph{collineation} is a one-to-one mapping $c$ on the points of the plane carrying 
lines into lines. Singer proved that in $\PG(2,q)$ there is always a collineation $c$ that 
cyclically permutes the $q^2+q+1$ points. Then if we label the points $P_i$, $0\leq i\leq q^2+q$
so that $P_i=c^i(P_0)$ for every $i$, then the indicies corresponding
to points on the same line 
will form a difference set of size $q+1$ in the additive cyclic group 
$\mathbb{Z}_{q^2+q+1}$. For the other direction 
he remarks that such a difference set naturally induces a projective
geometry of order $q$. This strong connection motivates the name 'planar' difference set.


In general, for multiplicative groups $\cG_1,\cG_2$ two difference sets $\cD_1\subset \cG_1$ and $\cD_2\subset \cG_2$ are called
{\em equivalent} if there exists a group isomorphism 
$\varphi: \cG_1\rightarrow \cG_2$ and an element $\Gamma\in \cG_2$ such that 
$\varphi(\cD_1)= \Gamma \cdot \cD_2$. For example, in Abelian groups any difference set $\cD$
is equivalent to its inverse $\cD^{-1}$ via the isomorphism
$Y\rightarrow \frac{1}{Y}$. 
In Singer's construction from above, choosing different lines for the same 
collineation also results in equivalent difference sets.
Singer conjectured that every difference set in $\mathbb{Z}_{q^2+q+1}$
is equivalent to his construction.  
This conjecture is still very much 
open.
Berman~\cite{B} and Halberstam and Laxton~\cite{HL}, verifying a related
conjecture of Singer, determined the exact number of reduced difference
sets of Singer type in $\mathbb{Z}_{q^2+q+1}$, where a(n additive) difference set
is called {\em reduced} if it contains both $0$ and $1 \in
\mathbb{Z}_{q^2+q+1}$.

We finish this subsection with an equivalent formulation of Singer's
construction that will be useful later. We consider the cyclic group
$\cG=\raisebox{.25em}{$\F_{q^3}^{*}$}\left/\raisebox{-.25em}{$\F_q^{*}$}\right.$. The
order of this group is $q^2+q+1$ and the cosets of those elements
$A\in \F_{q^3}^{*}$ for which the trace $\Tr(A)=A+A^q+A^{q^2}=0$ form
a difference set of size $q+1$ which is equivalent to the Singer difference set (see e.g.~\cite{P}).

\subsection{Results}\label{results}

\paragraph{Difference sets.} 
In order to treat infinite difference sets as well, 
we introduce our definitions and results for arbitrary fields, 
restricting to finite fields only when necessary. This general 
approach also keeps the arguments more transparent.

Let $\F$ be an arbitrary field and $\K$ a Galois extension of degree 3. Let 
$\phi$ be a nonidentity $\F$-automorphism of $\K$. Then the Galois 
group of the extension $\K$ over $\F$ is $\{\phi, \psi, id \}$, where 
$\psi=\phi \circ \phi$ and, of course, $id=\phi\circ \phi \circ
\phi$.  We denote by $\N_{\K/\F}$ the norm 
of this extension: for $A \in \K $ we have 
$\N_{\K/\F}(A)=A \cdot \phi(A) \cdot \psi (A)\in \F$. When it causes no
confusion, which will be the case most of the time, we omit writing
the index $\K/\F$. Examples of such
extensions are simplest cubic fields~\cite{Sh}, which are important
and well studied objects in algebraic number theory. 

We shall consider two functions $h_1, h_2: \K \rightarrow \K$ defined as
\begin{align*} 
h_1(X)&=\phi(X)\cdot X+X+1& \text{and}& & h_2(X)&=\phi(X)\cdot X+\phi(X)+1.
\end{align*}
For $i=1,2$ let $\cH_i$ denote the set of roots of $h_i$ in $\K$ and
let $\cN$ denote the set of elements in $\K^*$ with norm $1$. 
It is easy to see that $\cN$ is a subgroup of the multiplicative group $\K^*$.
In our first result we prove that $\cH_1$ and $\cH_2$ form a
difference set in the cyclic group $\cN$ with an explicit formula for the mixed representation.

\begin{thm} \label{thm:main} 
The sets $\cH_1$ and $\cH_2$ are equivalent difference sets in the
group $\cN$ and $\cH_2 = \cH_1^{-1}$. 
Furthermore, the unique mixed representation of an element 
$A\in \cN\setminus \{ 1 \}$ (with respect to $\cH_i$) 
is given by the following explicit formulas:
\begin{align}\label{explicit}
A_1&=\frac{A\cdot \phi(A)-1}{1-\phi(A)} \in \cH_1 & \text{and}& &
                                                                  A_2&=\frac{A-A\cdot
                                                                       \phi(A)}{A\cdot
                                                                       \phi(A)-1}
                                                                       \in
                                                                       \cH_2. 
\end{align}
\end{thm}
Infinite difference sets were earlier constructed by Hughes~\cite{H}
using a greedy-like approach. 
Our construction 
is more explicit and so offers more possibilities to study these nice combinatorial structures.

Next we spell out the statement of our theorem for finite fields.
We shall show that in this case our difference sets are of Singer type. 
Let $\F=\F_q$, $\K=\F_{q^3}$ and 
take $\phi$ to be the Frobenius automorphism $X\rightarrow X^{q}$. Then 
$h_1(X)$ and $h_2(X)$ become polynomials over $\F_q$, namely 
\begin{align*}
h_1(X)&=X^{q+1}+X+1 & \text{and}& & h_2(X)&=X^{q+1}+X^q+1,
\end{align*}
and the group $\cN$ is the unique 
subgroup of order $q^2+q+1$ in $\F_{q^3}^*$.

\begin{cor} \label{corollary:main} 
The root set $\cH_1\subseteq \F_{q^3}$ of the polynomial
$h_1(X)=X^{q+1}+X+1\in \F_q[X]$ and the root set $\cH_2\subseteq
\F_{q^3}$ of the polynomial $h_2(X)=X^{q+1}+X^q+1\in \F_q[X]$ are
equivalent difference sets of size $q+1$ in the order $q^2+q+1$ cyclic group $\cN$
of norm $1$ elements of $\F_{q^3}$ with $\cH_2 = \cH_1^{-1}$. 
Furthermore the unique mixed representation of an element 
$A\in \cN\setminus \{ 1 \}$ (with respect to $\cH_i$)
is given by the following explicit formulas:
\begin{align*}
A_1&=\frac{A^{q+1}-1}{1-A^q} \in \cH_1 & \text{and}& &
                                                       A_2&=\frac{A-A^{q+1}}{A^{q+1}-1}
                                                            \in \cH_2. 
\end{align*}
Moreover, both $\cH_1$ and $\cH_2$ are equivalent to the Singer difference set. 
\end{cor}

Even though the difference sets given here are equivalent to 
Singer's construction, their description as the roots of a simple
polynomial and the explicit formulas for the mixed representation are
interesting on their own right.

In Section~\ref{sec:general} we generalize Theorem~\ref{thm:main} and
Corollary~\ref{corollary:main} from the planar case to Singer difference sets 
of arbitrary classical parameters.

\paragraph{Projective norm graphs.}

As already mentioned earlier, one of the main results in \cite{BMRSz} is that if $p>3$ then $\NG(q,4)$ contains $K_{4,6}$ as a subgraph. 
Relying heavily on the properties of the Singer difference set from
Corollary~\ref{corollary:main} and its connection to norm equation systems,
here we settle the remaining cases of characteristic $2$ and $3$.  

\begin{thm} \label{thm:K46} 
Let $q=p^k>4$ where $p\not\equiv 1 \ {\rm mod}\ (3)$ is a prime. 
Then  $\NG(q,4)$ contains $K_{4,6}$ as a subgraph.
In particular, $\NG(2^k,4)$ and $\NG(3^\ell, 4)$ contain $K_{4,6}$ as
a subgraph for $k\geq 3$ and $\ell \geq 2$.
\end{thm}

Complementing Theorem~\ref{thm:K46} we remark that it is immediate that $\NG(2,4)$ 
does not contain $K_{4,6}$, and in $\NG(3,4)$ and $\NG(4,4)$ we verified by
computer search that there is no $K_{4,6}$ either. 

\section{Proofs}\label{proofs}

\subsection{Properties of the sets \texorpdfstring{$\mathcal{H}_i$}{Lg}}

We start by proving a few helpful properties of the sets $\cH_i$, inlcuding
their unique mixed representation property from Theorem~\ref{thm:main}. 

\begin{prop} \label{prop:Hi}
\
\begin{enumerate}[(i)]
\item For $Y\in \K^*$ we have
  $h_2\left(\frac{1}{Y}\right)=\frac{h_1(Y)}{\Phi(Y)\cdot Y}$. In
  particular, $\cH_2= \cH_1^{-1}.$
\item Every $A\in \cN \setminus \{ 1 \}$ can be represented uniquely
  as a product $A=A_1\cdot A_2$ of an element $A_1$ of $\cH_1$ and an
  element $A_2$ of $\cH_2$.
  This representation is given by  
 \begin{align}\label{mixed-rep}
A_1&=\frac{A\cdot \phi(A)-1}{1-\phi(A)} & \text{and}& & A_2&=\frac {A-A\cdot \phi(A)}{A\cdot \phi(A)-1}. 
\end{align}
\item
\begin{equation*}
\cH_1\cap\cH_2= \cH_1 \cap \F = \cH_2 \cap \F = \{Y\in \F : Y^2+Y+1=0\},
\end{equation*} 
in particular, if $\F=\F_q$, $\K=\F_{q^3}$ then
\begin{equation*}
\cH_1\cap\cH_2 =\left\{ \begin{array}{cl} \{1\} & \mbox{if } q\equiv 0
                                                 \ \pmod 3 \\ 
\{\alpha,\alpha^{-1}\},\text{ where }\alpha^3=1,\alpha\neq 1 & \mbox{if } q\equiv 1 \ \pmod 3 \\ 
\emptyset & \mbox{if } q\equiv 2 \ \pmod 3 
\end{array}
\right..
\end{equation*} 
\item 
If $Y\in \K$ then $\phi(h_i(Y)) = h_i (\phi(Y))$, in particular $\phi(\cH_i) = \cH_i$.
\end{enumerate}
\end{prop}

\begin{proof} 
\noindent
(i) For $Y\in \K$ we have
\begin{align*} 
h_2\left( \frac 1Y \right) &=\frac{1}{Y}\cdot \phi\left(\frac{1}{Y}\right)+\phi\left(\frac{1}{Y}\right)+1=\frac{1}{Y \cdot \phi(Y)}+ 
\frac{1}{\phi(Y)}+1=\\
&=\frac{1+Y+Y \cdot \phi(Y)}{Y \cdot \phi(Y)}=
\frac{h_1(Y)}{Y \cdot \phi(Y)}.
\end{align*}
The rest of the statement follows by the definition of the $\cH_i$'s.

\medskip

\noindent
(ii) First we show the uniqueness of the representation of the form (\ref{mixed-rep}). Suppose that 
$A\in \cN\setminus \{ 1\}$ and $A =A_1 \cdot A_2$
with $A_1\in \cH_1$ and $A_2=A/A_1\in \cH_2$. Then
\begin{align*}
A_1 \cdot \phi(A_1)+A_1 +1&=0,\\
\frac{A}{A_1} \cdot \phi\left( \frac{A }{A_1}\right)+
\phi\left(\frac{A}{A_1}\right)+1&= 
\frac{A}{A_1} \cdot  \frac{\phi(A) }{\phi(A_1)} +
\frac{\phi(A)}{\phi(A_1)}+1=0.  
\end{align*}
By expressing $\phi(A_1)$ from the two equations we obtain 
\begin{equation*}
\frac{A_1+1}{A_1}=\frac{A \cdot \phi(A)}{A_1}
+\phi(A).
\end{equation*}
Solving for $A_1$ gives that
the only possibility is
\begin{equation*} 
A_1 =\frac{A \cdot \phi(A)-1}{1-\phi(A) }.
\end{equation*}
This gives uniqueness and the stated formula (\ref{explicit}) 
for $A_2=A/A_1$ as well. It remains to verify that $A_i\in \cH_i$.  We consider first $A_1$. Using 
\begin{equation*}
\N(A)=A\cdot \phi(A) \cdot \psi(A)=1
\end{equation*}
we have
\begin{align*}
h_1(A_1)&=\frac {A \cdot \phi(A)-1}{1-\phi(A)}\cdot \phi \left( \frac {A \cdot \phi(A)-1}{1-\phi(A)}\right)+
\frac {A \cdot \phi(A)-1}{1-\phi(A)}+1=\\
&= \frac {A \cdot \phi(A)-1}{1-\phi(A)}\cdot\frac {\phi(A) \cdot \psi(A)-1}{1-\psi(A)}+
\frac {\phi(A)(A -1)}{1-\phi(A)}\\
&=  \frac {\frac{1}{\psi(A)}-1}{1-\phi(A)} \cdot \frac {\frac{1}{A}-1}{1-\psi(A)}+
\frac {\frac{1}{A\cdot\psi(A)}(A -1)}{1-\phi(A)}=\\
&= \frac {1-A}{A \cdot \psi(A)
\cdot  (1-\phi(A))}+
\frac {A -1}{A \cdot \psi(A)
\cdot (1-\phi(A))}=0,
\end{align*}
and hence $A\in \cH_1$. The verification of $A_2\in \cH_2$ is a similar
calculation that we leave to the reader. 
%

\medskip

\noindent
(iii) 
To see the desired equalties of sets note that for an element $Y\in \K$ we have $h_1(Y)=h_2(Y)$ if and only $Y=\phi (Y)$, which happens if
and only if $Y\in \F$. 


If furthermore $\F=\F_q$, then for $3|q$ we have $X^2+X+1=
(X-1)^2$, while otherwise the non-trivial third-roots
of unity are in $\F$ if and only if $3|q-1$.

\medskip

\noindent
(iv) The statement is a direct consequence of the fact that $\phi$ is an automorphism that fixes $\F$.
\end{proof}

\subsection{Difference sets and a norm equation system}
After Proposition~\ref{prop:Hi}(ii), all we need in order to complete the proof
of Theorem~\ref{thm:main} is to show that $\cH_1,\cH_2 \subseteq \cN$.   
In our next result we are not only proving that the elements of
$\cH_1 \cup \cH_2$ indeed have norm $1$, but we are also able to characterize
them as the solutions in $\K$ of the norm equation system
\begin{align} \label{norm1}
\N(X)&=1, & \N(X+1)&=-1,
\end{align}
which is closely related to projective norm graphs.
Besides this connection, the system (\ref{norm1}) also arises naturally
in algebraic number theory. When $\mathbb K$ is a number field, then the algebraic integer solutions 
of (\ref{norm1}) will be exceptional units in the sense of Nagell \cite{N}. Exceptional units are interesting objects, in particular one can
show that their number is always finite. 

\begin{prop}\label{prop:diffsets-normequations}
$Y\in \K$ is a solution of (\ref{norm1}) if and only if $Y\in \cH_1 \cup \cH_2$.
\end{prop}
\begin{proof}
The proof is adapted from \cite{BMRSz} which handled the finite field setting.
We first write the equations of 
(\ref{norm1}) in simple product form.
\begin{align*}
 X \cdot \phi(X) \cdot \psi(X)&=1\\
(X+1) \cdot \phi(X+1) \cdot \psi (X+1)&=(X+1) \cdot (\phi(X)+1) \cdot (\psi (X)+1)=-1
\end{align*}
Expressing $\psi(X)$ from both equations and clearing denominators gives
%
\begin{equation*}
(X+1)\cdot (\phi(X)+1)= -X\cdot \phi(X)-
X\cdot \phi(X)\cdot (X+1)\cdot (\phi(X)+1).
\end{equation*}
After expanding and simplifying we obtain
\begin{equation} \label{phi-eq}
X^2\cdot (\phi(X))^2+X \cdot (\phi(X))^2+ X^2\cdot \phi(X)+ 3X\cdot \phi(X)
+\phi(X) +X+1=0.
\end{equation}
It is easy to check that the left hand side of (\ref{phi-eq}) is
\begin{equation*} 
h_1(X)\cdot h_2(X)=(\phi(X)\cdot X+X+1)\cdot (X\cdot \phi(X)+\phi(X)+1).
\end{equation*}
%

Clearly, by the above calculations, any solution $Y$ of (\ref{norm1}) will be a 
solution of (\ref{phi-eq}), hence $h_1(Y)\cdot h_2(Y)=0$, which in turn is equivalent to $Y\in \cH_1\cup\cH_2$.

Conversely, let $Y\in \cH_1\cup\cH_2$ i.e. $h_1 (Y) \cdot h_2(Y)=0$. Then either
\begin{equation*}
h_1(Y)=0, \text{ in which case } \phi(Y) =-\frac{Y+1}{Y}=:u(Y),
\end{equation*}
or
\begin{equation*}
h_2(Y)=0, \text{ in which case } \phi(Y) = -\frac{1}{Y+1}=:v(Y).
\end{equation*}
Note that the denominator is non-zero in both cases.
Using part (iv) of Proposition~\ref{prop:Hi}, in the first case we 
obtain
\begin{equation*}
\psi(Y)=\phi\circ \phi(Y)= u(u(Y)) =
-\frac{-\frac{Y+1}{Y}+1}{-\frac{Y+1}{Y}} 
= v(Y),
\end{equation*}
while in the second case we get
\begin{equation*}
\psi(Y) =\phi \circ \phi(Y)= v(v(Y)) =
-\frac{1}{-\frac{1}{Y+1}+1} = u(Y).
\end{equation*}
In both cases we have $\{\phi(Y), \psi(Y)\}=\{u(Y),
v(Y)\}$ and therefore
\begin{equation*}
\N(Y) = Y \cdot \phi(Y)\cdot \psi(Y) = 
Y \cdot u(Y)\cdot v(Y) = Y \cdot
\left( -\frac{1}{Y+1}\right) \cdot\left( -\frac{Y+1}{Y} \right) = 1.
\end{equation*}
Similarly, for the norm of $Y+1$ we obtain
\begin{align*}
\N(Y+1)&= (Y+1)\cdot \phi(Y+1)\cdot \psi(Y+1)= (Y+1)\cdot (\phi(Y)+1)\cdot (\psi(Y)+1)\\
&=(Y+1)\cdot (u(Y)+1)\cdot (v(Y)+1)\\
&= (Y+1) \cdot \left( -\frac{Y+1}Y+1\right)
\cdot\left(-\frac{1}{Y+1}+1\right) = -1.
\end{align*}
This finishes the proof. 
\end{proof}

\begin{proof}[Proof of Theorem~\ref{thm:main}.]
Proposition~\ref{prop:diffsets-normequations} shows that $\N(X)=1$ for every $X\in \cH_1
\cup \cH_2$, hence $\cH_1, \cH_2 \subseteq \cN$. The equality $\cH_2=\cH_1^{-1}$ is just part (i) of Proposition~\ref{prop:Hi}, while the unique mixed representation of an element $A\in \cN \setminus \{ 1\}$ is provided by part (ii) of Proposition~\ref{prop:Hi}.
\end{proof}

\begin{proof}[Proof of Corollary~\ref{corollary:main}.]
Most of the statements follow directly by applying
Theorem~\ref{thm:main} to the setting $\F=\F_q$, $\K=\F_{q^3}$, and
$\phi(X)=X^{q}$. Note that then $\cN$ is an order
$q^2+q+1$ multiplicative group, hence any difference set in it, in
particular $\cH_1$ and $\cH_2$, must have size $q+1$. 
In particular, this implies that both $h_1$ and $h_2$ have all their 
$q+1$ roots in $\cN\subseteq \F_{q^3}$.

It remains to prove the equivalence of $\cH_1$ (and so of $\cH_2$) to the Singer difference set $\cS$, for which we use the description
mentioned at the end of Subsection~\ref{subs:diffsets}, i.e.
\begin{equation*}
\cS=\{A\cdot \F_q^*\ :\ A\in \F_{q^3}^*,\ \Tr(A)=0\}\subseteq \raisebox{.25em}{$\F_{q^3}^{*}$}\left/\raisebox{-.25em}{$\F_q^{*}$}\right..
\end{equation*}
%

Consider the group homomorphism $\Phi : \F_{q^3}^* \rightarrow \F_{q^3}^*$, given
by $\Phi (A) =A^{q-1}$. Since $N(\phi( A)) = A^{q^3-1} = 1$, the image of
$\Phi$ is contained in $\cN$. The kernel is $\F^*$ and 
$|\cN| = q^2 +q +1 =\frac{q^3-1}{q-1}$, so the quotient
map $\bar{\Phi}$ provides an isomorphism between the groups
$\raisebox{.25em}{$\F_{q^3}^{*}$}\left/\raisebox{-.25em}{$\F_q^{*}$}\right.$
and $\cN$.


We claim that $\bar{\Phi}$ maps $\cS$ into
$\cH_1$, i.e. if $A\in \F_{q^3}^*$ is such that $\Tr(A)=A^{q^2}+A^q+A=0$, then $A^{q-1} \in \cH_1$.
Indeed, 
\begin{align*}
h_1(A^{q-1})&=A^{(q-1)(q+1)}+ A^{q-1}+1=\frac{A^{q^2}}{A}+ A^{q-1}+1=\\
&=\frac{(-A^q-A)}{A}+A^{q-1}+1=-A^{q-1}-1+A^{q-1}+1=0.
\end{align*}
Using the injectivity of $\bar{\Phi}$ 
and $|\cS|=|\cH_1|$, we obtain that $\bar{\Phi}(\cS)=\cH_1$, showing that $\bar{\Phi}$ gives  
the desired equivalence of the difference sets $\cS$ and $\cH_1$.
\end{proof}

\subsection{Norm equation systems with the maximum number of solutions}

Let $\F$ be an arbitrary field and $\K$ a Galois extension of degree $3$. Finding a $K_{4,6}$ in $\NG(\F,\K)$, up to a couple of technicalities to be
handled later, 
is essentially equivalent to finding distinct pairs $(A_1,a_1), (A_2,
a_2), (A_3, a_3)  \in \K \times \F^*$ such that 
the system
\begin{align}\label{3eqgen}
\N(X+A_1)&=a_1, & \NG (X+A_2)&=a_2, & \N(X+A_3)&=a_3,
\end{align}
has six solutions $X\in \K$.
Note that by the $K_{4,7}$-freeness of the projective norm graph $\NG (\F, \K)$ 
(see Subsection~\ref{infnorm} in the Appendix), we also know that (\ref{3eqgen})
can have {\em at most} six solutions for
any values of the parameters.

In what follows we will study this system for special sets of
parameters and try to find 
particular choices where the maximum possible six solutions are attained. 
The main result of this subsection is that in some cases this is indeed possible. For a fixed element $A\in \cN\setminus\{1\}$ consider the following
system of norm equations:
\begin{align}\label{3eq}
\N(X)&=1, & \N(X+1)&=-1, & \N(X+A)&=-1.
\end{align}
\begin{thm} \label{thm:equations2} 
Let $q=p^k>4$, where $p\not\equiv 1 \ ({\rm mod}\ 3)$ is prime, and
let $\F=\F_q$, $\K=\F_{q^3}$. 
Then there exists an element $A\in \cN \setminus \{1\}$ such that the system (\ref{3eq})
has 6 solutions in $\F_{q^3}$.
\end{thm}
The starting point in the proof of Theorem~\ref{thm:equations2} is an
observation, valid over arbitary fields, that connects the solution set of the system~(\ref{3eq}) to the solution set of (\ref{norm1}) and hence to the difference sets $\cH_1$, $\cH_2$. 

\begin{prop} \label{prop:solution-pair} 
An element $Y\in \K$ is a solution of (\ref{3eq}) with parameter $A
\in \cN \setminus \{ 1 \}$ if and only if 
$Y$ and $\frac{A}{Y}$ are both contained in $\cH_1 \cup \cH_2$. 
\end{prop}
\begin{proof} By Proposition~\ref{prop:diffsets-normequations} 
$Y$ and $\frac{A}{Y}$ are both contained in $\cH_1 \cup \cH_2$ if and
only they are both solutions of (\ref{norm1}). We show that this is
equivalent to $Y$ being a solution of (\ref{3eq}).

Suppose first that $Y\in \K$ is a solution of (\ref{3eq}).
Then a fortiori $Y$ is a solution of (\ref{norm1}) and $Y\ne 0$. Also,
\begin{align*}
\N\left( \frac{A}{Y}\right)&=\frac{\N(A)}{\N(Y)}=\frac{1}{1}=1,\\
\N\left( \frac{A}{Y}+1\right)&=\N\left(\frac{A+Y}{Y}\right) =
\frac{\N(A+Y)}{\N(Y)}=\frac{-1}{1}=-1,
\end{align*}
hence $\frac{A}{Y}$ is also a solution of (\ref{norm1}).

Conversely, assume that $Y$ and $\frac{A}{Y}$ are both solutions of (\ref{norm1}). 
Then, in particular, $Y$ satisfies the first two equations from (\ref{3eq}), as for
the third one we have
\begin{equation*}
\N(Y+A)= \N\left( Y\left(1+\frac{A}{Y}\right)\right)= \N(Y) 
\N\left(\frac{A}{Y}+1\right)=1\cdot -1=-1,
\end{equation*}
and hence $Y$ is a solution of (\ref{3eq}).
\end{proof}

By the previous propostion we will be looking for product
representations $Y \cdot \frac{A}{Y}$ of an
element $A$ from the set $\cH_1 \cup \cH_2$. To prove
Theorem~\ref{thm:equations2}  we actually need to
find an element $A \in \cN \setminus \{ 1\}$ that has 
three such product representations $A=B_1C_1 =B_2C_2=B_3C_3$, such that the six 
elements $B_1,C_1, B_2, C_2, B_3, C_3 \in \cH_1 \cup \cH_2$ are all
distinct. For this we will crucially use that $\cH_1$ and $\cH_2 =
\cH_1^{-1}$ are difference sets in $\cN$ and inverses of each other.

 Recall that if $\cD$ is any difference set in some multiplicative
 group $\cG$ then every element $A \in \cG\setminus\{1\}$ has a unique
 representation, called mixed representation, 
as a product $B\cdot C=A$ such that one of $B$ and $C$ 
is from $\cD$ and the other is from $\cD^{-1}$.
In the next propositions we summarize our knowledge about other product representations.
To this end we will call a product $B\cdot C =A$ a {\em $\cD$-representation} of 
the element $A\in \cG$ if both $B$ and $C$ are from $\cD$. 

\begin{prop} \label{prop:diffset-representations}
Let $\cD$ be an arbitrary difference set in some multiplicative group $\cG$.
Then every $A\in \cG$ has at most one $\cD$-representation. 
\end{prop}
\begin{proof}
Let us assume that  $D_1D_2 =D_3D_4$ for some $D_1, D_2, D_3, D_4 \in \cD$.
Then $D_1D_3^{-1} = D_4D_2^{-1}$, and this, by the difference set property, 
is either $1$ or we have $D_1=D_4$ and
$D_2 =D_3$: in any case $\{D_1, D_2\} = \{ D_3, D_4\}$. 
\end{proof}

The explicit descriptions of our difference sets allow us to characterize when an
$\cH_i$-representation with distinct factors exists.

\begin{prop} \label{prop:representations}
\
\begin{enumerate}[(i)]
\item If $char (\F) \neq 2$ then $A\in \cN$ has an
  $\cH_1$-representation with different factors if and only if  $(A+1
  - A\phi(A))^2 - 4A$ is a non-zero square in $\K$. 
\item If $char (\F) \neq 2$ then $A\in \cN$ has an
  $\cH_2$-representation with different factors if and only if
  $(A\phi(A) + \phi(A) -1)^2 -4A\phi(A)$ is a non-zero square in $\K$. 
\end{enumerate}
\end{prop}


\begin{proof}
In the case of finite fields most parts of the argument have already 
appeared in \cite{BMRSz}. 
As the proof of the two parts is analogous, below we only present the one of (i).

Suppose first that $A\in \cN$ has an $\cH_1$-representation. This means that there is an element $Y\in \K$ such that $Y$ and $\frac{A}{Y}$ are both roots of $h_1$, i.e. 
\begin{align*}
h_1(Y)&=\phi(Y)\cdot Y+Y+1=0 & \text{and} & & h_1\left(\frac{A}{Y}\right)&=\phi\left(\frac{A}{Y}\right)\cdot \frac{A}{Y}+\frac{A}{Y}+1=0.
\end{align*}
after expressing $\phi(Y)$ from both equations, letting them being equal and clearing denominators we obtain $Y^2+\left(A+1-a\phi(A)\right)\cdot Y+A=0$. Clearly, the role of $Y$ and $\frac{A}{Y}$ can be switched, which means that both $Y$ and $\frac{A}{Y}$ are roots of the quadratic equation 
\begin{align}\label{quadratic_h1}
X^2+\left(A+1-A\phi(A)\right)\cdot X+A=0.
\end{align}
If char($\F)\neq 2$, this is possible only if the discriminant 
\begin{equation*}
D=\left(A+1-A\phi(A)\right)^2-4A
\end{equation*}
is a nonzero square in $\K$. 

For the other direction suppose that $D=\left(A+1-A\phi(A)\right)^2-4A$ is a nonzero square in $\K$, i.e. there is some element $G\in \K^*$ such that $D=G^2$. Then we know that the quadratic equation in (\ref{quadratic_h1}) has two different roots, namely $X_{\pm}=\frac{A\phi(A)-A-1\pm G}{2}$. Clearly, $X_{+}\cdot X_{-}=A$, so to finish the proof it is enough to show that $X_{\pm}\in \cH_1$, i.e. $h_1(X_{\pm})=0$.

Using $N(A)=A\phi(A)\psi(A)=1$ we have
\begin{align*}
\phi(D)&=\phi\left(\left(A+1-A\phi(A)\right)^2-4A\right)=\left(\phi(A)+1-\phi(A)\psi(A)\right)^2-4\phi(A)=\\
&=\left(\phi(A)+1-\frac{1}{A}\right)^2-4\phi(A)=\frac{1}{A^2}\left(\left(A\phi(A)+A-1\right)^2-4A^2\phi(A)\right)= \\
&=\frac{1}{A^2}\left(\left(A+1-A\phi(A)\right)^2-4A\right)=\frac{1}{A^2}D.
\end{align*}
Then $\left(\frac{\phi(G)}{G}\right)^2=\frac{\phi(G^2)}{G^2}=\frac{\phi(D)}{D}=\frac{\frac{1}{A^2}D}{D}=\frac{1}{A^2}$ and hence $\frac{\phi(G)}{G}=\pm\frac{1}{A}$. However $N\left(\frac{\phi(G)}{G}\right)=\frac{\phi(G)}{G}\cdot\phi\left(\frac{\phi(G)}{G}\right)\cdot\psi\left(\frac{\phi(G)}{G}\right)=\frac{\phi(G)}{G}\cdot \frac{\psi(G)}{\phi(G)}\cdot\frac{G}{\psi(G)}=1$ which, as char($\F)\neq 2$, excludes $\frac{\phi(G)}{G}=-\frac{1}{A}$. Therefore $\phi(G)=\frac{1}{A}G$. As a consequance we get
\begin{align*}
\phi(X_{\pm})&=\phi\left(\frac{A\phi(A)-A-1\pm G}{2}\right)=\frac{\phi(A)\psi(A)-\phi(A)-1\pm \phi(G)}{2}=\\
&=\frac{\frac{1}{A}-\phi(A)-1\pm\frac{1}{A}G}{2}=\frac{1}{A}X_{\pm}+\frac{1-A\phi(A)}{A}.
\end{align*}
Now we are ready to substitute $X_{\pm}$ into $h_1$.
\begin{align*}
h_1(X_{\pm})&=X_{\pm}\cdot \phi(X_{\pm})+X_{\pm}+1=X_{\pm}\cdot\left(\frac{1}{A}X_{\pm}+\frac{1-A\phi(A)}{A}\right)+X_{\pm}+1=\\
&=\frac{1}{A}\left(X_{\pm}^2+\left(A+1-A\phi(A)\right)\cdot X_{\pm}+A\right)=0,
\end{align*}
where at the last equality we just used that $X_{\pm}$ are the roots
of (\ref{quadratic_h1}).
\end{proof}

For the elements of $\cH_i$ the existence of a $\cH_{3-i}$-representation follows directly.

\begin{prop}\label{prop:unique}
If $A\in \mathcal{H}_i$ then its unique $\cH_{3-i}$-representation is 
given by $\frac{1}{\phi(A)} \cdot \frac{1}{\psi(A)} = A$.
\end{prop}
\begin{proof}
On the one hand,  as $A\in \mathcal{H}_i\subseteq \mathcal{N}$, we have that
$A= \frac{N(A)}{\phi(A)\cdot \psi(A)} = \frac{1}{\phi(A)}\cdot \frac{1}{\psi(A)}$. On the
other hand,  by Proposition~\ref{prop:Hi}(i) we have $
\frac{1}{A}\in \mathcal{H}_{3-i}$, and so by
Proposition~\ref{prop:Hi}(iv) we have $\frac{1}{\phi(A)} =
\phi\left(\frac{1}{A}\right) \in \cH_{3-i}$ and $\frac{1}{\psi(A)} =
\phi\left(\frac{1}{\phi(A)}\right) \in \mathcal{H}_{3-i}$.
The uniqueness follows from
Proposition~\ref{prop:diffset-representations}.
\end{proof}
From now on we will consider the special case $\F=\F_q$ (and
$\K=\F_{q^3}$), and separate the cases according to the characteristic
modulo $3$.

\subsubsection{Characteristic \texorpdfstring{$2$}{Lg}}



In this subsection we settle the case of characteristic $2$ which was
left open in \cite{BMRSz}. Our method extends to odd characteristic $p\equiv 2
\pmod 3$, which was settled, in a much stronger form, already in
\cite{BMRSz}.

First we show the existence of six solutions when $q$ is congruent to $2$ modulo $3$
(as opposed to $p$).

\begin{prop} \label{prop:equations} 
Let $q > 2$ be a prime power such that $q\equiv 2 \ {\rm mod}\ (3)$,
and let  $\F=\F_q$ and $\K=\F_{q^3}$. 
Then there exists an element $A\in \cN \setminus \{ 1 \}$ such that the system (\ref{3eq})
has 6 solutions in $\F_{q^3}$.  
\end{prop}
\begin{proof} 
We will find an element $A \in \cN\setminus \{1\}$ that has an $\cH_i$-representation
$A = B_i\cdot C_i$ for both $i=1$ and $2$, such that $B_i\neq
C_i$. These four elements, together with the two elements $A_1\in \cH_1$
and $A_2\in \cH_2$ from the unique mixed representation of $A$ (which exists
by Proposition~\ref{prop:Hi}(ii)) give us six
distinct solutions of (\ref{3eq}). Indeed, they are solutions of (\ref{3eq}) by
Proposition~\ref{prop:solution-pair} and their distinctness follows immediately from the
fact that the sets $\cH_1$ and $\cH_2$ are disjoint by Proposition~\ref{prop:Hi}(iii).

In order to find the appropriate element $A$, for a set $\cH\subseteq \cN$ we define 
$$\cH^*:=\{B\cdot C\ :\ B,C\in \cH,\ B\ne C\}$$ to be the set of its
pairwise products from distinct factors, and show that $\cH_1^* \cap \cH_2^*$  is not empty. 

First note that by Proposition~\ref{prop:representations}(ii)
the pairwise products of the elements of $\cH_i$ 
are all distinct, hence 
the cardinality of $\cH_i^*$ is $\binom{|\cH_i|}{2} =
\frac{q^2+q}{2}$. 
Both $\cH_1^*$ and $\cH_2^*$ are subsets of the $(q^2+q)$-element set $\cN
\setminus \{ 1\}$. For this note that $\cH_i$ is contained in $\cN$,
which is closed under multiplication, and that $1 \not \in \cH_i^*$
since $\cH_i$ is disjoint from its inverse 
$\cH_{3-i}$ by our assumption on $q$ and Proposition~\ref{prop:Hi}(iii). 
Hence the only way  $\cH_1^*$ and $\cH_2^*$ could be disjoint is if their union is $\cN \setminus
\{ 1 \}$. In this case however it would also hold that
\begin{equation} \label{sum}
\sigma(\cN\setminus \{1\})=\sigma(\cH_1^*)+\sigma(\cH_2^*),
\end{equation}
where $\sigma(\mathcal{S})$ denotes the sum of the elements of a subset 
$\mathcal{S}\subseteq \F_{q^3}$.  On the one hand 
the set $\cN$ is the collection of all $q^2+q+1$ roots in $\F_{q^3}$ of the polynomial
$X^{q^2+q+1}-1$ and so $\sigma(\cN)$ is $(-1)$ times the 
coefficient of $X^{q^2+q}$ in this polynomial, which is $0$. From this we
obtain that $\sigma(\cN\setminus \{1\})=-1$. On the other hand $\cH_i$ is the set of
all $q+1$ roots in $\F_{q^3}$ of the polynomial $h_i(X)$, hence
$\sigma(\cH_i^*)$ is the coefficient of $X^{q-1}$ in $h_i(X)$, which is 0 for $q>2$. We
arrived to a contradiction, as the left hand side of (\ref{sum}) is (-1), while the
right hand side is $0$. 
\end{proof}

\begin{proof}[Proof of Theorem~\ref{thm:equations2} for $p \equiv 2 \pmod
  3$] 
First note that $\F_{q^6}$ is a cubic extension of $\F_{q^2}$ and the norm of an element $B\in \F_{q^3}\subseteq
  \F_{q^6}$ is the same, irrespective in which of the two fields we
  compute it: $\N_{\F_{q^3}/\F_q} (B) = \N_{\F_{q^6}/\F_{q^2}} (B)$. 
This means that if for an element $A\in \F_{q^3}$ with $\N_{\F_{q^3}/\F_q}(A) = 1$ the system
(\ref{3eq}) with the norm function $\N_{\F_{q^3}/F_q}$ has six distinct
solutions $X_1, \ldots , X_6 \in \F_{q^3}$, then the very same six
elements are also solutions of the the system 
(\ref{3eq})  with the norm function $\N_{\F_{q^6}/F_{q^2}}$. 

Now let $p \equiv 2 \pmod 3$ be a prime and let $q=p^k=p^{\ell2^m}$ be an
arbitrary power where $\ell$ is odd. 
Then Proposition~\ref{prop:equations} gives the required six distinct
solutions when $m=0$ and $p^\ell > 2$. 
By repeated application of the above observation, 
the statement also follows for any positive integer $m$ and $p^\ell >2$. These include
all the powers when $p > 2$. 

When $p=2$ then only those powers are included where $\ell \geq 3$.
So we are left with prime powers of the form $2^{2^m}$.
To settle these last cases one first resolves the problem when the prime power is
$2^{2^2} =16$ and then uses the above squaring trick to deduce the
case of arbitrary $2^{2^m} = 16^{2^{m-2}}$.

For $q=16$ we have found the appropriate $A \in \F_{16^3}$ of norm $1$, for which the
system (\ref{3eq}) has six distinct solution with the aid of a computer. 
To describe this example, let $U$ be the primitive element of $\F_{16^3}^*$ whose minimal
polynomial over $\F_{2}$ is $X^{12} + X^7 + X^6 + X^5 + X^3 + X + 1$, and consider the
system
\begin{align*} 
\N(X)&=1, & \N(X+1)& =1, & \N(X+U^{405})&=1.
\end{align*}
By Magma Calculator~\cite{BCP} it is easily verified that $A=U^{405}$ is in $\cN\setminus \{1\}$, and that the system 
has indeed six solutions, namely  $U^{1725}, U^{2775}, U^{3435}\in \cH_1$ and  
$U^{1065}, U^{2130}, U^{2370}\in \cH_2$ with
\begin{equation*}
A=U^{1065}\cdot U^{3435}= U^{1725}\cdot U^{2775}=U^{2130}\cdot
U^{2370}.
\end{equation*}
%
\end{proof}

\subsubsection{Characteristic \texorpdfstring{$3$}{Lg}}

\begin{proof}[Proof of Theorem~\ref{thm:equations2} for $p =3$] 
Just like in the previous subsection, we will find an element
$A \in \cN\setminus \{ 1\}$ which has an $\cH_i$-representation
$A = B_i\cdot C_i$ for both $i=1$ and $2$, such that these four
elements and the two elements $A_1 \in \cH_1$
and $A_2\in \cH_2$ from the mixed representation of $A$ (which exists
by Proposition~\ref{prop:Hi}(ii)) are pairwise
distinct and hence provide six distinct solutions of (\ref{3eq}). 

We do this in two steps. First we find an element that has both
$\cH_1$- and $\cH_2$-representation, but in one of them the factors
are not distinct. 

\begin{lemma} \label{lemma:5-solutions}
For $i=1$ or $2$ there is an element
$C\in \cH_i \setminus \{ 1 \}$, such that $C^2$ has an 
$\cH_{3-i}$-representation $C^2=B\cdot E$ with distinct factors $B\neq E$. 
\end{lemma}

Let us fix elements $C\in \cH_i \setminus \{ 1 \}$ and $B, E \in
\cH_{3-i}$ guaranteed by Lemma~\ref{lemma:5-solutions}. 
We show that the element $A:=\frac{C}{E}$ is the kind we are looking
for. First observe that by Proposition~\ref{prop:Hi}(i)
\begin{equation*}
A = C\cdot \frac{1}{E} = B\cdot \frac{1}{C}
\end{equation*} 
 provide a $\cH_i$- and $\cH_{3-i}$-representation of $A$, respectively.
Note furthermore that as $(\cH_i \setminus \{ 1 \}) \cap \cH_{3-i} =
\emptyset$, we have $C\neq E$ and hence $A \neq 1$. Consequently, by  
Proposition~\ref{prop:Hi}(ii) there exists a unique mixed
representation $A=A_1\cdot A_2 $ with $A_i\in \cH_i$.

Next we show that
these six elements from the representations are all distinct.
 
\begin{lemma} \label{lemma:distinctness}
The elements $A_i, C, \frac{1}{E} \in \cH_i$ and $A_{3-i}, \frac{1}{C}, B
\in \cH_{3-i}$ are all distinct.  
\end{lemma}
To finish the proof of Theorem~\ref{thm:equations2} note that 
by Proposition~\ref{prop:solution-pair} these elements provide six distinct solutions of (\ref{3eq}). 
\end{proof}

We finish this subsection proving the two lemmas from the above proof. 

\begin{proof}[Proof of Lemma~\ref{lemma:distinctness}] 
For the distinctness first we establish that none of the six elements
is $1$.
This is certainly true for $C$ and $\frac{1}{C}$ by the
choice of $C$ in Lemma~\ref{lemma:5-solutions}. Now assume that $B$ or
$E$ is $1$, say $B=1$ (the argument in the case $E=1$ is analogous). 
On the one hand, as $E \in \cH_{3-i}$, by Proposition~\ref{prop:unique} $E$ has a unique
$\cH_i$-representation: $E = \frac{1}{E^q} \cdot \frac{1}{E^{q^2}}$. 
On the other hand $E= B\cdot E = C\cdot C$ is also a $\cH_i$-representation of $E$, 
so by the uniqueness we must have $E^q = \frac{1}{C} = E^{q^2} = \phi(E^q)$, and thus
$E^q\in \F_q$ and $E^q =\frac{1}{C} \in \cH_{3-i}$. 
That means $\phi(E)=E^q \in \F_q\cap \cH_{3-i} = \cH_1 \cap \cH_2 = \{ 1 \}$
by Proposition~\ref{prop:Hi}(iii),  which is only possible if $E=1$. 
This contradicts $C\neq 1$ and implies $B, E\neq 1$. 
Finally assume that $A_1$ or $A_2$ is equal to $1$, say $A_1=1$. 
Then $A_1\cdot A_2 = A = B\cdot \frac{1}{C}$ are two
$\cH_2$-representations of $A$. By uniqueness either $B$ or $C$ should
be $1$, which is a contradiction by the above.

Since none of the six elements is $1$ and by part (iii) of
Proposition~\ref{prop:Hi} $\cH_1\cap \cH_2=\{1\}$, we established that
\begin{equation*}
\displaystyle \{A_i,C,\frac{1}{E}\}\cap \{A_{3-i},B,\frac{1}{C}\}=\emptyset.
\end{equation*}

We are left to show that $\left|\{A_i,C,\frac{1}{E}\}\right|=3$ and
$\left|\{A_{3-i},B,\frac{1}{C}\right|=3$. Since they are proved
analogously we present just the first one. 

If $A_i=C$, then $A_{3-i}=\frac{1}{E}$, which is a contradiction as
$A_{3-i}\in \cH_{3-i}$ and $\frac{1}{E}\in \cH_i$ and none of them is $1$.
If $A_i=\frac{1}{E}$, then $A_{3-i}=C$, which is a contradiction similarly
as $A_{3-i}\in \cH_{3-i}$ and $C\in \cH_i$  and none of them is $1$.
Finally, suppose that $C=\frac{1}{E}$. Then we have $B=C^2\cdot
\frac{1}{E}=C^3$, which is a contradiction as 
$B\in \cH_{3-i}\setminus\{1\}$  and $C^3\in \cH_i\setminus\{1\}$ because the polynomial $h_i$ is defined over 
$\F_3$, hence if $h_i(C)=0$, then $h_i(C^3)=0$ as well.
\end{proof}

\begin{proof}[Proof of Lemma~\ref{lemma:5-solutions}] 
We want to find a $C\in \cH_i\setminus \{ 1 \}$ for $i=1$ or $2$, such that $C^2$
has an $\cH_{3-i}$-representation with different elements $B$ and $E$. 
This happens exactly if one of the formulas in 
Proposition~\ref{prop:representations}
is a nonzero square in $\F_{q^3}$ when we substitute $A=C^2$.
It turns out that after simplifying the substituted formula of (i)
using $h_2(C) =C^{q+1}+C^q+1=0$ and clearing its square denominator
we obtain the very same expression $D(C)$ as after simplifying the
substituted formula of (ii) using $h_1(C)=C^{q+1}+C+1=0$ and clearing its square denominator:
\begin{align*}
D=D(C)=&\left(C^2(C+1)^2+(C+1)^2-C^2\right)^2-4(C+1)^4C^2=\\
=&\left(C^2+3C+1\right)\cdot\left(C^2+C+1\right)\cdot\left(C^2+C-1\right)\cdot\left(C^2-C-1\right).
\end{align*}

We aim to find an element $C\in \cH_1\cup\cH_2 \setminus \{ 1 \}$ for which $D$  is a square in 
$\F_{q^3}$, or equivalently, $N(D)$ is a square in $\F_q$.

As it turns out the factors of $N(D)$ can be conveniently expressed
using the trace $Tr(C) = C + C^q + C^{q^2} =: t$ of $C$:
\begin{align*}
&N(C^2+3C+1) =-t^2-3t -1,\\
&N(C^2+C+1) =t^2 + 3t +9,\\
&N(C^2+C-1) =t^2+3t+1,\\
&N(C^2-C-1) =-t^2-3t-1, 
\end{align*}
and so
\begin{equation*}
N(D)= (-t^2 -3t -1)^2 (t^2+3t +9) (t^2+3t +1).
\end{equation*}




%
In characteristic $3$ this expression  is a square if and only if
$t^2+1$ is a square. Using Theorem 5.18 from \cite{LN} we get that
\begin{equation*}
\sum_{y\in \F_q}\eta(y^2+1)=-\eta(1)=-1,
\end{equation*}
where $\eta$ is the quadratic character of $\F_q$. Therefore, as $y^2 +1 =0$ has at most two solutions, for at least
$\displaystyle \frac{q-3}{2}$ elements $y\in\F_q$ the expression 
$y^2+1$ is a square.

This ensures the existence of many good ``traces'', which we can use
to construct many good $C$, as we now show that 
the trace function is a $3$-to-$1$ function on $\cH_1 \cup \cH_2
\setminus \{ 1 \}$. That
is, if $Tr(C_1) = Tr(C_2)$ for $C_1, C_2 \in \cH_1 \cup \cH_2\setminus\{1\}$
then $C_1$ and $C_2$ are conjugates of each other. 
For this we note that the minimal polynomial of an element $C$ of $\cH_1 \cup
\cH_2\setminus \{ 1 \}$ can be expressed just by the trace $t$ of $C$: 
 \begin{equation*}
m_t(X)=(X-C)(X-C^q)(X-C^{q^2}) = X^3- tX^2- (t+3)X-1.
\end{equation*}
Hence if $C_1$ and $C_2$ have the same trace, then they have the same
minimal polynomial.

Consequently there are exactly $\frac{|\cH_1 \cup \cH_2 \setminus \{ 1
  \}|}{3} = \frac{2q}{3}$ elements in $\F_q^*$ that are traces of some
element in $\cH_1 \cup \cH_2 \setminus \{ 1 \}$. 

In conclusion, there are at least $\left(\frac{2q}{3} + \frac{q-3}{2}\right)-q =
\frac{q-9}{6} > 0$ elements $t\in F_q$ which are traces of an element 
$C$ from $\cH_1 \cup \cH_2 \setminus \{ 1 \}$, and for which $t^2+1$,
and hence also $D=D(C)$, is a square. 
This completes the proof for $q >9$. 

Otherwise, by our assumption on
$q$, we are left with the case $q = 9$. Then we can directly find a $t\in \F _9$ such that it is the trace of an element 
$C\in \cH_1 \cup \cH_2 \setminus \{ 1 \}$ and $t^2+1$ is a nonzero square in $\F_9$. In fact $t=-1$ will do. First note that $t^2+1=1+1=2$ is in $\F_3$ and hence is a square in its quadratic extension $\F_9$. Now, to finish the argument it is enough to show that $m_{-1}(X)=X^3+X^2+X-1$ is the minimal polynomial of some $C\in \cH_1 \cup \cH_2 \setminus \{ 1 \}$ over $\F_9$, because then we automatically have $\Tr(C)=t=-1$. It is immediate that $m_{-1}(X)$ is irreducible over $\F_3$, hence it can not have a root in $\F_9$, and therefore it is irreducible over $\F_9$. Next consider the polynomial $h_1(X)$ for $q=3$. Then (when computed over $\F_3$) we have
\begin{equation*}
h_1(X)=X^4+X+1=(X-1)(X^3+X^2+X-1)=(X-1)m_{-1}(X),
\end{equation*}
which in view of Proposition~\ref{prop:diffsets-normequations} means that the roots of $m_{-1}(X)$ solve the system (\ref{norm1}) with norm function $\N_{\F_{3^3}/\F_3}$. But then, as we have seen earlier, the roots of $m_{-1}(X)$ also solve the system (\ref{norm1}) with norm function $\N_{\F_{9^3}/\F_9}$, and hence, again by Proposition~\ref{prop:diffsets-normequations} and the fact $m_{-1}(1)\neq 0$, we have that 
any root of $m_{-1}$ is in $\cH_1 \cup\cH_2\setminus\{1\}$, as desired. 
\end{proof}

\subsection{Norm equation systems and \texorpdfstring{$K_{4,6}$}{Lg} subgraphs in \texorpdfstring{$\NG(q,4)$}{Lg}}

In this subsection we prove Theorem~\ref{thm:K46}. In the first lemma we
connect solutions of a norm equation systems with three equations 
to four adjacencies in the projective norm-graph. For this let again $\F$ be an arbitrary field and $\K$ a Galois extension of degree $3$.

\begin{lemma} \label{lemma:harmas}
Suppose that for $A_1, A_2, A_3 \in \K^*$ and $a_1, a_2, a_3 \in \F^*$, the
 element $Y\in \K^*$ is a solution of (\ref{3eqgen}). 
Then in the projective norm graph $\NG(\F,\K)$ the vertex
$\left(\frac{1}{Y}, \frac{1}{\N(Y)} \right)$ is adjacent to all the vertices
\begin{align*}
&\left(\frac{1}{A_1}, \frac{a_1}{\N(A_1)}\right),&
&\left(\frac{1}{A_2}, \frac{a_2}{\N(A_2)}\right),& 
&\left(\frac{1}{A_3}, \frac{a_3}{\N(A_3)}\right),&
&(0,1).
\end{align*}
\end{lemma}  

\begin{proof} 
We check the adjacenies from the statement. 
For the first three vertices, using $\N(Y+A_i)=a_i$, we have
\begin{equation*}
\N\left(\frac{1}{A_i}+\frac{1}{Y} \right) 
= \N\left(\frac{Y+A_i}{A_iY} \right)=
\frac{\N(Y+A_i)}{\N(A_iY)}=\frac{a_i}{\N(A_i)}\cdot \frac{1}{\N(Y)}.
\end{equation*}
For the last vertex we have
\begin{equation*}
\N\left(0+\frac{1}{Y} \right) 
= \N\left(\frac{1}{Y} \right)=1\cdot \frac{1}{\N(Y)},
\end{equation*}
as requested. 
\end{proof}

The statement of Theorem~\ref{thm:K46} now follows easily.

\begin{proof}[Proof of Theorem~\ref{thm:K46}]
By Theorem~\ref{thm:equations2} there exists an element 
$A\in \cN\setminus \{1\}$ 
such that the system (\ref{3eq}) admits $6$ solutions in $\F_{q^3}$. 
We choose an element $C\neq 0, -1, -A$ from $\F_{q^3}$ 
such that $\N(C) \neq 1$.
This is clearly possible since $q^3 - 3 -(q^2+q+1) > 0$ for $q>2$. 
Let $A_1 = C$, $A_2 = C+1$, $A_3 = C+A$, $a_1=1$, $a_2=-1$, and 
$a_3=-1$. The corresponding norm equation system 
of the form (\ref{3eqgen}), as the translation of (\ref{3eq}) 
by $C$, also has six distinct solutions, which we denote by $Y_1, \ldots, Y_6$. 
Since by the choice of $C$ the elements $A_1, A_2,$ 
and $A_3$, as well as the solutions $Y_1, \ldots , Y_6$ are non-zero, 
Lemma~\ref{lemma:harmas} is applicable. As the elements $A_1, A_2, A_3$ are also distinct, so are the vertices 
\begin{align*}
& \left(\frac{1}{A_1}, \frac{a_1}{N(A_1)}\right),&
&\left(\frac{1}{A_2}, \frac{a_2}{N(A_2)}\right),& 
&\left(\frac{1}{A_3}, \frac{a_3}{N(A_3)}\right),&
&(0,1).
\end{align*}
which we denote by $(B_i, b_i)$, $i=1,2,3,4,$.
By Lemma~\ref{lemma:harmas} each of them is adjacent to the six distinct
vertices $\left(\frac{1}{Y_j}, \frac{1}{N(Y_j)} \right)$, $j=1, \ldots
, 6$. 

In order for these $24$ adjacencies to 
indeed give rise to a $K_{4,6}$ in $NG(q,4)$, 
we need to make sure that none of them represents a loop.

If $p=2$, then $\NG(q,4)$ has no loop edges, so all the participating 
$4+6$ vertices are different, and they form a $K_{4,6}$.

Assume now that $p > 2$. Let us call an element $D\in \F_{q^3}$ {\em
  bad} for the pair of indicies $(i,j)$, if $B_i+D=\frac{1}{Y_j}-D$. 
For every fixed pair $(i,j)$ there exists exactly one $D$
which is bad for $(i,j)$, namely $D=\frac{1}{2}\left(\frac{1}{Y_j} -
  B_i\right)$.  Therefore, together there
are at most $24$ elements $D$ which are bad for some pair $(i,j)$. 
As $q^3\geq 27 >24$, there exists an element $D\in \F_{q^3}$ 
which is not bad for any pair. For this $D$, the 
ten vertices $(B_i+D, b_i)$, $i=1, 2, 3, 4$, and
$\left(\frac{1}{Y_j}-D, \frac{1}{N(Y_j)}\right)$, $j=1,
\ldots , 6$, are all distinct and hence give rise to a $K_{4,6}$.
\end{proof}

\section{General difference sets of Singer type} \label{sec:general}

In general an $(n,m,\lambda)$ difference set is a set ${\cal D}$ of
$m$ elements in a multiplicative group ${\cal G}$ of order $n$, such
that any element
$g \in {\cal G}$ has exactly $\lambda$ mixed product representations with respect to $\cD$. 
Singer's construction naturally generalizes to the case with parameters 
$(n,m,\lambda)=(\frac{q^{t}-1}{q-1},\frac{q^{t-1}-1}{q-1},\frac{q^{t-2}-1}{q-1})$
for any $t\geq 3$, which are called Singer parameters. 
Unlike for $t=3$, the planar case, for many values $t > 3$ difference sets
having Singer parameters yet being inequivalent 
to Singer's construction are known to exist (see e.g.~\cite{GMW}).

In the case of planar difference sets we made use of a convenient
description of the Singer difference set inside
the multiplicative group
$\raisebox{.25em}{$\F_{q^3}^{*}$}\left/\raisebox{-.25em}{$\F_q^{*}$}\right.$
as the collection of cosets of elements of trace $0$. 
In what follows, we extend this description to the general setting of
degree $t$ cyclic Galois extensions over arbitrary fields $\F$,
that we encountered already in Section~\ref{sec:intro} in connection with projective
norm graphs.
 To this end, let $\F$ be an arbitrary field and $\K$ a cyclic
Galois extension of degree $t$ with $t\geq 3$. 
Let $\Tr_{\K/\F}$ be the trace function corresponding to this
extension, i.e. for $A\in \K$ we have 
\begin{equation*}
\Tr_ {\K/\F}(A)=A+\phi(A)+\phi^{(2)}(A)+\cdots+\phi^{(t-1)}(A),
\end{equation*}
where  $\phi^{(i)}$ denotes the $i$-fold iteration of some generator $\phi$ of
the Galois group.
Now consider the subset
\begin{equation*}
\mathcal{S}_t=\{\overline{Y}\ |\ Y\in \K^*,\ \Tr_{\K/\F}(Y)=0\}
\end{equation*} 
of the multiplicative group $\raisebox{.25em}{$\K^*$}\left/\raisebox{-.25em}{$\F^{*}$}\right.$, where $\overline{Y}$ denotes the image of $Y\in \K^*$ under the natural map $\K^*\rightarrow \raisebox{.25em}{$\K^*$}\left/\raisebox{-.25em}{$\F^{*}$}\right.$. 

It is known \cite{P} that when $\F = \F_q$ is a finite field then
$\mathcal {S}_t$ is a Singer difference set with parameters
$(\frac{q^{t}-1}{q-1},\frac{q^{t-1}-1}{q-1},\frac{q^{t-2}-1}{q-1})$. 
This can be extended for arbitrary $\F$, using the natural $(t-1)$-dimensional projective space structure on
$\raisebox{.25em}{$\K^*$}\left/\raisebox{-.25em}{$\F^{*}$}\right.$ 
(which is induced by the $t$-dimensional $\F$-vector space structure
of $\K$).
It turns out that for any non-identity element $\overline{A}$, 
the set of $\mathcal{S}_t$-elements from the mixed representations of
$\overline{A}$ with respect to $\mathcal{S}_t$ forms a subspace 
of projective dimension $t-3$. 

\begin{prop}\label{prop:subspace}
Let $\overline{\cL}$ be a subspace of  $\raisebox{.25em}{$\K^*$}\left/\raisebox{-.25em}{$\F^{*}$}\right.$ of projective dimension $t-2$. Then for every element $\overline{A}\in \raisebox{.25em}{$\K^*$}\left/\raisebox{-.25em}{$\F^{*}$}\right.\setminus\left\{\overline{1}\right\}$ the set 
\begin{equation*}
\cR_{\overline{\cL}}\left(\overline{A}\right)=\{\overline{B}\in
\overline{\cL}\ |\ \exists\ \overline{C}\in \overline{\cL}\text{ such
  that }\overline{A}=\overline{B}\overline{C}^{-1}\} = \{\overline{B}\in \overline{\cL}\ |\ \overline{B}/\overline{A} \in \overline{\cL} \}
\end{equation*}
forms a subspace of projective dimension $t-3$. 
\end{prop}
\begin{proof}
Let $\cL$ be a $(t-1)$-dimensional subspace of $\K$ over $\F$ such that $\overline{\cL}=\raisebox{.25em}{$\cL\setminus\{0\}$}\left/\raisebox{-.25em}{$\F^{*}$}\right.$. Then for any element $\overline{A}\in \raisebox{.25em}{$\K^*$}\left/\raisebox{-.25em}{$\F^{*}$}\right.\setminus\left\{\overline{1}\right\}$ we have 
\begin{equation*}
\cR_{\overline{\cL}}\left(
  \overline{A}\right)=\raisebox{.25em}{$\cR_{\cL}\left( A\right)\setminus\{0\}$}\left/\raisebox{-.25em}{$\F^{*}$}\right.,
\end{equation*}
where $A\in \K^*\setminus \F^*$ is such that $\overline{A}=A\cdot \F^*$ and 
\begin{equation*}
\cR_{\cL}\left( A\right)=\{B\in \cL\ |\ \exists\ C\in \cL\text{ such that }A=BC^{-1}\}.
\end{equation*}
Observe that  $\cR_{\cL}\left( A\right)=\cL\cap A\cL$, where $A\cL=\{AL\ | \ L\in \cL \}$. 

Since $A\not\in \F^*$, the $(t-1)$-dimensional subspaces $\cL$ and
$A\cL$ are different, 
hence their intersection has dimension $t-2$. Therefore the projective dimension of $\cR_{\overline{\cL}}\left(\overline{A}\right)$ is indeed $t-3$.
\end{proof}

\begin{cor}
Let $\overline{\cL}$ be a subspace of $\raisebox{.25em}{$\F_{q^t}^*$}\left/\raisebox{-.25em}{$\F^{*}$}\right.$ of projective dimension $t-2$. Then $\overline{\cL}$ is a difference set in the 
multiplicative group~\ $\raisebox{.25em}{$\F_{q^t}^*$}\left/\raisebox{-.25em}{$\F^{*}$}\right.$ 
with parameters \linebreak $(\frac{q^{t}-1}{q-1},\frac{q^{t-1}-1}{q-1},\frac{q^{t-2}-1}{q-1})$. 
In particular, so is $\mathcal{S}_t$.
\end{cor}

\begin{proof}
The set $\cR_{\overline{\cL}}\left(\overline{A}\right)$ is in a
one-to-one correspondence with the collection of mixed
$\overline{\cL}$-respresentations of $\overline{A}$.
Since a projective space of dimension $t-i$ over $\F_q$ has size
$\frac{q^{t-i+1}-1}{q-1}$, the parameters of the difference set
follow.

Finally note that as $\Tr_{\K/\F} : \K\rightarrow \F$ is a non-trivial
$\F$-linear function, $\text{Ker}(\Tr_{\K/\F})$ is a
$(t-1)$-dimensional subspace of $\K$. Hence the 
set\linebreak $\cS_t=\raisebox{.25em}{$\text{Ker}(\Tr_{\K/\F})\setminus\{0\}$}\left/\raisebox{-.25em}{$\F^{*}$}\right.$
is a subspace of
$\raisebox{.25em}{$\K^*$}\left/\raisebox{-.25em}{$\F^{*}$}\right.$ of
projective dimension $t-2$, and as such is a difference set with
Singer parameters.
\end{proof}


\smallskip

In Corollary~\ref{corollary:main} we gave an alternative description of planar
difference sets of Singer type over finite fields as the root set of a simple
polynomial and gave explicit formulas of the product representation
of each element. 
Here we extend this result to the general setting. For this let $\N_{\K/\F}$ denote the norm function of the extension $\K/\F$ and $\cN$ the group of elemens of norm $1$ in $\K$. Furthermore, we shall consider the function
\begin{align*}
d_t(X)&=1+X+X \phi(X)+X \phi(X)\phi^{(2)}(X)+\cdots+X\phi(X)\cdots \phi^{(t-2)}(X).
\end{align*}
We remark that the function $d_t(X)$ appears in a paper of
Foster~\cite{F} in the formulation of the Murphy condition.

In the next theorem we show that the set $$\cD_t=\left\{A\in \K\ |\
  d_t(A)= 0 \right\}$$ of roots of $d_t(Y)$ in $\K$ is contained in
the multiplicative group $\cN$ and has the same difference set property that the
mixed representation of any element $A\in \cN\setminus\{1\}$ with respect to $\cD_t$
form (in some sense) a projective space of dimension $t-3$ over $\F$. In addition we will also be able to describe concisely these product representations.

\begin{thm}\label{thm:genmain}
There is a group isomorphism $\overline{\Phi}:
\raisebox{.25em}{$\K^{*}$}\left/\raisebox{-.25em}{$\F^{*}$}\right.\rightarrow\cN$
such that $\overline{\Phi}(\cS_t)=\cD_t$. In particular, through $\overline{\Phi}$, the set $\cD_t$ inherits the difference set property of $\cS_t$ just like the projective space structure. Moreover, given an element $A\in \cN\setminus\{1\}$ the different mixed representations of $A$ with respect to $\cD_t$ are exactly the products $B\cdot\left(\frac{B}{A}\right)^{-1}$, where $B$ is a root in $\K$ of the function
\begin{align*}
f_{t,A}(X)&=d_t(X)-\frac{1}{\phi^{(t-1)}(A)} \cdot d_t\left(\frac{X}{A}\right).
\end{align*}
\end{thm}

\begin{proof}
%
Consider the $\K^*\rightarrow \K^*$ map $\Phi$ defined by
$X\rightarrow \frac{\phi(X)}{X}$. On the one hand, one readily sees
that the map $\Phi$ maps $\K^*$ into $\cN$. On the other hand, by
Hilbert's Theorem 90~\cite{L} we know that for every $A\in \cN$ there
is an element $Y\in \K^*$ such that $A=\Phi(Y)$, which in turn shows
that $\Phi$ is surjective. Therefore, as Ker$(\Phi)=\F^*$, 
the quotient map $\overline{\Phi}: \raisebox{.25em}{$\K^{*}$}\left/\raisebox{-.25em}{$\F^{*}$}\right.\rightarrow\cN$ provides an isomorphism between the respective groups. 

Next we show that the image of $\cS_t$ under the map $\overline{\Phi}$ is $\cD_t$. For this let $Y\in \K^*$. Then, on the one hand, we have
\begin{align*}
d_t(\Phi(Y))&=d_t\left(\frac{\phi(Y)}{Y}\right)=1+\sum_{j=0}^{t-2}\prod_{i=0}^j\frac{\phi^{(i+1)}(Y)}{\phi^{(i)}(Y)}=\\
&=1+\sum_{j=0}^{t-2}\frac{\phi^{(j+1)}(Y)}{Y}=\frac{1}{Y}\Tr_{\K/\F}(Y).
\end{align*}
Therefore, if $\overline{Y}\in \cS_t$, then $\overline{\Phi}\left(\overline{Y}\right)\in \cD_t$.  Finally, let $A\in \cD_t$, i.e. $A$ is a root of $d_t$. Then $\N_{\K/\F}(A)=1+A\cdot \phi(d_t(A))-d_t(A)=1$, and hence, again by Hilbert's Theorem 90, there is an element $Y\in \K^*$ such that $A=\Phi(Y)$ and so $A=\overline{\Phi}\left(\overline{Y}\right)$. By the above calculations  $\Tr_{\K/\F}(Y)=Y\cdot d_t(\Phi(Y))=Y\cdot d_t(A)=0$, meaning that $\overline{Y}\in \cS_t$. This concludes the proof of $\overline{\Phi}(\cS_t)=\cD_t$.

Now let us turn to the second part of the theorem.  
Given an element $A\in \cN\setminus\{A\}$ first take a mixed representations $A=B\cdot C^{-1}$ with $B,C\in \cD_t$. Then, in particular, we have $C=\frac{B}{A}\in \cD_t$ and hence $d_t(B)=d_t\left(\frac{B}{A}\right)=0$. Therefore, we have that $B$ is also a root of the function
\begin{align*}
f_{t,A}(X)&=d_t(X)-\frac{1}{\phi^{(t-1)}(A)} \cdot d_t\left(\frac{X}{A}\right).
\end{align*}
For the other direction suppose that $B$ is a root of $f_{t,A}$. Note that then necessarily $B\neq 0$, as otherwise we would have $0=f_{t,A}(0)=1-\frac{1}{\phi^{(t-1)}(A)}$, which is a contradiction as for $A\in \cN\setminus\{1\}$ we have $1-\frac{1}{\phi^{(t-1)}(A)}\neq 0$.
To finish the proof we need to show that $B,\frac{B}{A}\in D_t$, as then the product $B\cdot\left(\frac{B}{A}\right)^{-1}$ is a valid product representation of $A$ with respect to $\cD_t$. Using that $N_{\K/\F}(A)=1$ and $\phi^{(t)}\equiv \text{id}$, we have
\begin{align*}
0&=\phi(0)=\phi(f_{t,A}(B))=\phi(d_t(B))-\frac{1}{\phi^{(t)}(A)}\cdot \phi\left(d_t\left(\frac{B}{A}\right)\right)\\
&=\frac{1}{B}\cdot \left(d_t(B)+\N_{\K/\F}(B)-1\right)-\frac{1}{A}\cdot\frac{1}{\frac{B}{A}}\cdot\left(d_t\left(\frac{B}{A}\right)+\N_{\K/\F}\left(\frac{B}{A}\right)-1\right)\\
&=
\frac{1}{B}\left(d_t(B)-d_t\left(\frac{B}{A}\right)\right)\ \Longrightarrow\ d_t(B)=d_t\left(\frac{B}{A}\right),\\
\end{align*}
and hence $0=f_{t,A}(B)=d_t(B)\left(1-\frac{1}{\phi^{(t-1)}(A)}\right)$. However, as remarked earlier, for $A\in \cN\setminus\{1\}$ we have $1-\frac{1}{\phi^{(t-1)}(A)}\neq 0$, so this at once implies that $d_t(B)=d_t\left(\frac{B}{A}\right)=0$, and so $B,\frac{B}{A}\in D_t$, as required.
\end{proof}

Next we spell out the special case of Theorem~\ref{thm:genmain} when
$\F=\F_q$, $\K=\F_{q^t}$ and $\phi$ is the Frobenius automorphism
$X\rightarrow X^q$. This gives a description of the classic Singer
difference set inside $\cN$ as the set of roots of a simple polynomial and describes the
mixed representations of any element also using the roots of a
polynomial.

\begin{cor}\label{cor:genmain}
Let $q=p^k$ be a prime power, $t\geq 3$ an integer, and let us define over $\F_q$
the polynomial
\begin{equation*}
d_t(Y)=1+Y+Y^{1+q}+Y^{1+q+q^2}+\cdots+Y^{1+q+\cdots +q^{t-2}}
\end{equation*}
of degree $\frac{q^{t-1}-1}{q-1}$. Then the set  
\begin{equation*}
\cD_t=\left\{A\in \F_{q^t}\ |\ d_t(A)= 0 \right\}
\end{equation*}
of roots of $d_t(Y)$ forms a $(\frac{q^{t}-1}{q-1},\frac{q^{t-1}-1}{q-1},\frac{q^{t-2}-1}{q-1})$-difference
set in the cyclic group $\cN$ of norm $1$ elements of $\F_{q^t}$, which
is equivalent to the Singer difference set $\cS_t$. Moreover, given an element $A\in \cN\setminus\{1\}$, the $\frac{q^{t-2}-1}{q-1}$ different mixed $\cD_t$-representations of $A$ are exactly the products $B\cdot\left(\frac{B}{A}\right)^{-1}$, where $B$ is a root in $\F_{q^t}$ of the degree $\frac{q^{t-2}-1}{q-1}$ polynomial
\begin{align*}
f_{t,A}(X)&=d_t(X)-A^{1+q+\cdots+q^{t-2}}\cdot d_t\left(\frac{X}{A}\right)
\end{align*}
\qed
\end{cor}

In connection with Corollary~\ref{cor:genmain} first note, that in particular it implies that the polynomials $d_t(X)$ and
$f_{t,A}(X)$ always split over $\F_{q^t}$. Also, in the special case
$t=3$, we recover the difference set $\cH_1$ from 
Corollary~\ref{corollary:main}. In this case the polynomial $f_{t,A}$ is linear and its unique root is exactly the elment $A_1$ from Corollary~\ref{corollary:main}.

\section{Concluding remarks}

It is widely conjectured~\cite{Bol,CG} that
the KST upper bound, 
\begin{equation}
\label{eq:KST}
\ex(n, K_{t,s}) \leq c_{t,s} n^{2-1/t},
\end{equation}
is tight up to constant factor for every $s\geq t\geq 2$. 
Together with the results of \cite{BMRSz}, Theorem~\ref{thm:K46}
establishes that for $t=4$, the projective norm graph 
$\NG(q,4)$ does not resolve this
conjecture beyond the cases $s\geq 7$. 

Several interesting questions remain open.

\paragraph{Complete bipartite graphs in projective norm graphs.}

Both in \cite{BMRSz} and in Theorem~\ref{thm:K46} we could 
only find a special kind of copies of $K_{4,6}$. The number of these
copies is only roughly $q^7$. If that was it, then a simple uniform 
random subgraph of $\NG(q,4)$ with a
few deleted edges would prove the tightness of the
KST-bound for $t=4$ and $s=6$. 
 We think however, also supported by 
computer experiments, that the number of copies of $K_{4,6}$ in $\NG(q,4)$ 
should be the same order as their typical number in the
random graph of the same edge density. 
\begin{conjecture} 
The number of copies of $K_{4,6}$ in $\NG (q,4)$ is $\Theta (q^{16})$.
\end{conjecture}

The determination of $s(t)$ is still widely open for $t\geq 5$, 
when we do not even know whether there is
a $K_{t,t}$ in $\NG(q,t)$ for every large enough $q$. While it is
probably more realistic to expect that there are copies of 
$K_{t,  (t-1)!}$ for every $t\geq 5$ and large enough $q$ (besides
numerology, i.e. that $s(t) = (t-1)!$ for $t=2, 3$ and $4$, there are
also algebro-geometric heuristics pointing towards this), we 
harbour a slim hope that $t=4$ was still a special case. 
At least the graph $\NG(q, 4)$ seems quite special, with a unique structure and 
symmetries, and maybe that alone is responsible for the presence of $K_{4,6}$ subgraphs.  

\paragraph{Infinite projective norm graphs.} 

The first constructions of dense $K_{t,t}$-free graphs
were motivated by simple facts from real Euclidean geometry: 
two lines of the plane intersect in at most one point; three unit
spheres in $3$-space intersect in at most two points.  
Consequently the point/line incidence graph of the Euclidean
plane is $K_{2,2}$-free, and the unit-distance graph of the Euclidean
$3$-space is $K_{3,3}$-free. Furthermore these infinite $K_{t,t}$-free graphs 
are ``dense'' in terms of the dimension of the neighborhoods. So when 
defined over appropriate finite fields, in a way that the algebra in the proof of their
$K_{t,t}$-freeness carries over, their number of edges verifies 
the tightness of the KST-bound 
for $t=2$ and $t=3$.

The (projective) norm graphs were not constructed this way, yet one can
define them over an arbitrary field $\F$, see Section~\ref{sec:intro}. 
The key lemma from \cite{KRSz} holds over any field, which implies
that the proof of the $K_{t,(t-1)!+1}$-freeness of $\NG(q,t)$ 
also extends to $\NG(\F, \K)$ for arbitrary $\F$ and arbitrary Galois extension $\K$ of degree $t-1$. (For completeness we include the  
argument in the Appendix.) In particular, if $t=4$ then $\NG(\F, \K)$ does not contain
$K_{4,7}$ for any field $\F$. After seeing that $\NG(q, 4)$ does
contain $K_{4,6}$ for any $q>4$, it might seem
plausible 
that the same is 
true for infinite fields.

\paragraph{The tightness of the KST-bound.}
The tightness of the order of magnitude of the KST-bound is a central question of the area.
This conjecture suggests that whatever density is not ruled out by
simple double counting, 
should essentially be possible to realize with a construction. Here we speculate
that this might not be the case and offer a counter-conjecture.

In any graph with $cn^{7/4}$ edges, the number of common neighbors of
an average $4$-tuple is (at least) a constant $c'$ depending on $c$. 
If this graph with $cn^{7/4}$ edges is random then this constant average 
is spread out over $\binom{n}{4}$ distributions that are each approximately
Poisson with mean $c'$. Consequently for any $s$, a positive constant
proportion of $4$-tuples have at least $s$ neighbors. 
In contrast, in any $K_{4,s}$-free construction with
$cn^{7/4}$ edges (matching the KST-bound), 
no $4$-tuple can have more than $s-1$ neighbors. 
So in such constructions each of the Poisson-tails has to be absorbed by the
$4$-tuples with at most $s-1$ common neighbors.  
Should such graphs exist for some $s$, they must be extremely rare, their mere existence 
has to be a coincidence and should require quite a bit of structure.

In all known constructions (including Klein~\cite{E}, Brown~\cite{Br},
KRS~\cite{KRSz}, ARS~\cite{ARSz} and Bukh~\cite{Bukh}) this is realized using the algebro/geometric notion of
dimension and its strong correlation with the cardinality of the
corresponding variety: an ``everyday'' $d$-dimensional variety over
$\F_q$ has roughly $\Theta (q^{d})$ points. 
To achieve that the common neighborhood of 
four vertices is less than a constant $s$, one appeals to the geometric 
intuition that in the four-dimensional space the intersection of four
hypersurfaces, that are in general
enough position, is $0$-dimensional, and hence it is the union of constantly
many points. 
A graph can be defined on a four-dimensional 
space of roughly $q^4 = : n$ vertices, and the neighborhood of
each vertex can be chosen to be some hypersurface, which then have
roughly the desired size $q^3 = n^{3/4}$.
For a $K_{4,s}$-free graph the intersection of any four 
of the neighborhood-hypersurfaces should have size $< s$.
Now if the neighborhood-hypersurfaces are carefully chosen, so that any four of
them are in general enough position, then their intersection is
$0$-dimensional and hence has size $\Theta (q^0)$, a constant. 

How to choose the hypersurfaces and what is this constant? 
Even though choosing randomly is a generally good strategy 
(witnessed by the random algebraic construction of Bukh~\cite{Bukh}),
finding good explicit choices, as it is often the case, is not so straightforward. 
By the KST-bound the constant bounding the
neighborhoods of $t$-tuples in any graph with $cn^{2-1/t}$ edges is at least
$t-1$, and the projective norm graph chooses neighborhoods where they
are bounded by not more than $(t-1)!$. The current analysis of the random choice
gives an upper bound of $t^{O(t)}$. 

Now how small could this constant be? 
We believe that the presence of some notion of ``dimension'' in this
problem is a necessity and this constant is just going to be in the
nature of the geometry of the hypersurface-neighborhoods we have
chosen. As such, it will not just be limited by the simple
combinatorial restrictions of the KST-bound but also by those of 
geometry/algebra. And then its extrema should be delivered by a regular, rigid 
structure with distinctive properties. For $t=4$ we have seen ample evidence
that the projective norm graph $\NG(q,4)$ fits this bill, and 
tend to accept it as the limit of what algebra can offer in this
realm. Since we know now that $K_{4,6}$ does occur in $\NG(q,4)$, we 
conjecture the following.

\begin{conjecture}\label{con:counter-kst}
$\ex(n, K_{4,6}) = o(n^{7/4})$. 
\end{conjecture}

We note that should this conjecture be true, it of course implies that the KST-bound
is not tight for the symmetric case $K_{4,4}$ either. That further 
implies that $\ex( n, K_{t,t}) = o(n^{2-1/t})$ for every $t\geq
5$; this is the consequence of (an adaptation of) a theorem of Erd\H os and Simonovits~\cite{ES}.

While we do believe Conjecture~\ref{con:counter-kst}, at the same time we also think that it is
more likely that we see it disproved than proved. 
For a proof 
one might need to develop a two step approach.
Given a $K_{4,4}$-free graph with $cn^{7/4}$ edges, 
build up a significant-enough proportion of a
pseudo-algebraic/geometric framework using the
neighborhoods as hypersurfaces, with surfaces having appropriate
intersection sizes and structure. Then, provided the pseudo-algebra/geometry 
gives a structure rigid enough, establish the existence of a
$K_{4,4}$. Preliminary results in this direction 
were proven by Blagojevic, Bukh, and Karasev~\cite{BBK} and in this paper.


\section{Appendix}

\subsection{\texorpdfstring{$K_{4,7}$}{Lg}-freeness of \texorpdfstring{$\NG (\F, \K)$}{Lg}}\label{infnorm}

As before, let $\F$ be an arbitrary field and for $t\geq 2$ let $\K$ be a cyclic Galois extension of degree $t-1$, whose Galois group is generated by the automorphism $\phi$. We aim to prove that the projective norm graph $\NG(\F,\K)$ is $K_{t,(t-1)!+1}$-free, i.e. that the common neighbourhood of $t$ vertices has size at most $(t-1)!$. The proof is exactly the same as in the case of finite fields, and it is included here only because we need some of the steps anyway.

For $\ell \geq 2$ let $U=\left\{(B_i,b_i)\mid 1\leq  i\leq \ell \right\}$ be an $\ell$-set of vertices. Without loss of generality we may assume that $B_i\neq B_j$ for $i\neq j$, as otherwise the common neighbourhood of $U$ would be empty. For $1\leq  i\leq \ell -1$ let
\begin{align*}
A_i &= A_i(U) := \frac{1}{B_i-B_\ell} \in \K & \text{ and } & & a_i &= a_i(U):=\frac{b_i}{b_\ell} \cdot \N(B_i) \in \F^*,
\end{align*}
and consider the system
\begin{equation}\label{neighboureqB}
\N\left(X + A_i\right) = a_i \quad 1\leq  i\leq \ell -1
\end{equation}
of norm equations. Simple substitutions show that a vertex $(X,x)$ is in the common neighbourhood of $U$ if and only if $\xi\left((X,x)\right):=\frac{1}{X+B_{\ell}}$ is a solution to (\ref{neighboureqB}). Note that $\xi\left((X,x)\right)$ is well-defined, as $\N(X+B_{\ell})=b_{\ell}\cdot x\neq 0$. Therefore, we have that the size of the common neighbourhood of $U$ is at most the number of solutions to (\ref{neighboureqB}). We remark that it can be less (by one) exactly if $0$ is a solution to (\ref{neighboureqB}), which happens if $\ a_1=a_2=\cdots=a_{\ell}$.

Now set $\ell=t$ and let us recall the key lemma from the original proof due to Koll\'ar, R\'onyai and Szab\'o \cite{KRSz}.

\begin{lemma}\label{keylemma}
Let $\K$ be a field and $\alpha_{i,j},\beta_j\in \K$ for $1\leq i,j\leq m$ such that $\alpha_{i_1,j}\neq \alpha_{i_2,j}$ for $i_1\neq i_2$. Then the system of equations
\begin{equation*}
(x_1-\alpha_{i,1})(x_2-\alpha_{i,2})\cdots(x_m-\alpha_{i,m})=\beta_i \quad  i=1,2,\dots,m\end{equation*}
has at most $m!$ solutions $(x_1,x_2,\dots,x_m)\in \K^m$.
\end{lemma}

To finish the proof of the $K_{t,(t-1)!+1}$-freeness we just need to realize that the equations in (\ref{neighboureqB}) can be rewritten, namely for $1\leq i\leq t-1$ we have 
\begin{align*}
\N(X+A_i)&=(X+A_i)\cdot\phi(X+A_i)\cdots\phi^{(t-2)}(X+A_i)=\\
&=(X+A_i)\left(\phi(X)+\phi(A_i)\right)\cdots\left(\phi^{(t-2)}(X)+\phi^{(t-2)}(A_i)\right)=a_i.
\end{align*}
Now if for $1\leq  i,j\leq t -1$ we set $\alpha_{i,j}=\phi^{(j-1)}(A_i)$, $x_j=\phi^{(j-1)}(X)$ and $\beta_{i}=a_i$ then Lemma~\ref{keylemma} applies with $m=t-1$ and gives that the system (\ref{neighboureqB}) has at most $(t-1)!$ solutions and hence any $t$-set of vertices in $\NG(\F,\K)$ has at most $(t-1)!$ common neighbours, as desired.


\begin{thebibliography}{1}
\bibitem{ARSz} N. Alon, L. R\'onyai, T. Szab\'o. Norm-graphs: variations and applications.  Journal of Combinatorial Theory, Series B {\bf 76}:280--290 (1999).
\bibitem{BMRSz} T. Bayer, T. M\'esz\'aros, L. R\'onyai, T. Szab\'o. Exploring projective norm graphs. arXiv:1908.05190 (2019)
\bibitem{B} G. Berman. Finite projective plane geometries and difference sets. Transactions of the American Mathematical Society {\bf 74}:492--499 (1953).
\bibitem{BBK} P. Blagojevi\'c, B. Bukh, R. Karasev. Tur\'an numbers for $K_{s,t}$ -free graphs: topological obstructions and algebraic constructions. Israel Journal of Mathematics {\bf 197}(1):199--214 (2013).
\bibitem{Bol} B. Bollob\'as.  Extremal Graph Theory. Academic Press (1978).
\bibitem{BCP} W. Bosma, J. Cannon, C. Playoust. The Magma algebra system I -The user language.  Journal of Symbolic Computation {\bf 24}:235--265  (1997).
\bibitem{Br} W.G. Brown. On graphs that do not contain a Thomsen graph.  Canadian Mathematical Bulletin {\bf 9}:281--285 (1966).
\bibitem{Bukh} B. Bukh. Random algebraic construction of extremal graphs. Bullettin of the London Mathematical Society {\bf 47}(6):939--945 (2015).
\bibitem{CG} F.R.K. Chung, R. Graham. Erd\H{o}s on Graphs - His Legacy of Unsolved Problems. A.K. Peters, Wellesley (1998).
\bibitem{E} P. Erd\H{o}s. On sequences of integers no one of which divides the product of two others and related problems.  Mitt. Forsch. Institut Mat. und Mech. Tomsk {\bf 2}:74--82 (1938).
\bibitem{ES} P. Erd\H{o}s, M. Simonovits. A limit theorem in graph theory. Studia Scientiarum Mathematicarum Hungarica {\bf 1}:215--235 (1966).
\bibitem{F} L. Foster. HT90 and “simples” number fields. Illinois Journal of Mathematics {\bf 55}(4):1621--1655 (2011).
\bibitem{GMW} B. Gordon, W.H. Mills, L.R. Welch. Some new difference sets. Canadian Journal of Mathematics {\bf 14}:614--625 (1962).
\bibitem{G} C. Grosu. A note on projective norm graphs. International Journal of Number Theory {\bf 14}(1):55-62 (2018).
\bibitem{HL} H. Halberstam, R.R. Laxton. Perfect difference sets. Glasgow Mathematical Journal {\bf 6}(4):177-184 (1964).
\bibitem{H} D.R. Hughes. A note on difference sets. Proceedings of the American Mathematical Society {\bf 6}(5):689--692 (1955).
\bibitem{KRSz} J. Koll\'ar, L. R\'onyai, T. Szab\'o. Norm-graphs and bipartite Tur\'an numbers. Combinatorica {\bf 16}:399--406 (1996).
\bibitem{KST} T. K\H{o}v\'ari, V.T. S\'os, P. Tur\'an. On a problem of K. Zarankiewicz. Colloquium Mathematicae {\bf 3}(1):55--57 (1954).
\bibitem{L} S. Lang. Algebra. Springer-Verlag, New York (2002).
\bibitem{LN} R. Lidl, H. Niederreiter. Introduction to finite fields and their applications. Cambridge University Press (1986).
\bibitem{MP} E.H. Moore, H.S. Pollatsek. Difference sets: Connecting Algebra, Combinatorics, and Geometry. Student Mathematical Library {\bf 67}, American Mathematical Society (2013).
\bibitem{N} T. Nagell. Sur un type particulier d'unit\'es alg\'erbraiques. Arkiv f\"or Matematik {\bf 8}(2):163--184 (1970).
\bibitem{P} A. Pott. Finite Geometry and Character Theory. Lecture Notes in Mathematics {\bf 1601}, Springer-Verlag, Berlin, Heidelberg (1995).
\bibitem{Sh} D. Shanks. The Simplest Cubic Fields. Mathematics of Computation {\bf 28}(128):1137--1152 (1974).
\bibitem{S} J. Singer. A Theorem in Finite Projective Geometry and Some Applications to Number Theory. Transactions of the America Mathematical Society {\bf 43}(3):377--385 (1938).
\end{thebibliography}
\end{document}